\documentclass{amsart}
\usepackage{amsthm,amstext,amsmath,amscd,amssymb,latexsym}
\usepackage{color}
\usepackage[matrix,arrow]{xy}
\usepackage{amsfonts,enumerate,array,hyperref}

\newcommand{\N}{{\mathbb N}}
\renewcommand{\P}{{\mathbb P}}

\newcommand{\Z}{{\mathbb Z}}

\newcommand{\Sym}{{\rm{Sym}}}

\newcommand{\ko}{{\mathcal O}}

\newcommand{\s}{\mathcal}

\newcommand{\sI}{{\s I}}

\newcommand{\punkt}{\HHspace{-.3ex}\raise.15ex\HHbox to1ex{\HHuge.}}

\DeclareMathOperator{\im}{im}

\DeclareMathOperator{\rank}{rank}
\DeclareMathOperator{\Seg}{Seg}

 \newcommand{\paper}{: \begin{it}}
\newcommand{\jour }{, \end{it}}

\newtheorem{theorem}{Theorem}[section]

\newtheorem{proposition}[theorem]{Proposition}
\newtheorem{corollary}[theorem]{Corollary}

\theoremstyle{definition}

\newtheorem{example}[theorem]{Example}

\theoremstyle{remark}
\newtheorem{remark}[theorem]{Remark}

\numberwithin{equation}{section}

\begin{document}

\title[New examples of defective Segre-Veronese varieties]
{
New examples of defective secant varieties of Segre-Veronese varieties
}
\thanks{The first author is partly supported by NSF grant DMS-0901816.}

\author{Hirotachi Abo}
\address{Department of Mathematics, University of Idaho, Moscow, ID 83844, USA}
\email{abo@uidaho.edu}
\author{Maria Chiara Brambilla}
\address{Dip.\ di Scienze Matematiche, Universit\`a Politecnica delle Marche, Ancona, Italy}
\email{brambilla@dipmat.univpm.it }

\begin{abstract}
We prove the existence of defective secant varieties of three-factor and four-factor Segre-Veronese varieties embedded in certain multi-degree. These defective secant varieties were previously unknown and are of importance in the classification of defective secant varieties of  Segre-Veronese varieties with three or more factors. 
\end{abstract}
\maketitle
\section{Introduction}
\label{sec:intro}
Let $X$ be a non-degenerate projective variety in projective space $\P^N$ and let $p_1, \dots, p_s$ be linearly independent points of $X$. Then the $(s-1)$-plane $\langle p_1, \dots, p_s\rangle$ spanned by $p_1, \dots, p_s$ is called a {\it secant $(s-1)$-plane} to $X$. The Zariski closure of the union of all secant $(s-1)$-planes to $X$ is called the {\it $s^{\mathrm{th}}$ secant variety} of $X$ and denoted by $\sigma_s(X)$.  A basic question about secant varieties is to find their  dimensions.  A simple dimension count indicates
\[
\dim \sigma_s(X) \leq \min\{ s (\dim X+1)-1, \ N\}. 
\]
If equality holds, we say that $\sigma_s(X)$ {\it has the expected dimension}. Otherwise $\sigma_s(X)$ is said to be {\it defective}.  

In $1995$, Alexander and Hirschowitz~\cite{AH} finished classifying all the defective secant varieties of Veronese varieties. This work completed the Waring-type problem for polynomials, which had remained unsolved for some time.  There are corresponding conjecturally complete lists of defective secant varieties for Segre varieties (see \cite{AOP}) and for Grassmann varieties (see \cite{K, BDD}). Very recently, the defectivity of two-factor Segre-Veronese varieties was systematically studied in~\cite{AB}, where we suggested that secant varieties of two-factor Segre-Veronese varieties are not defective modulo a fully described list of exceptions. However, the secant defectivity of more general Segre-Veronese varieties is still less well-understood. 

In this paper we explore higher secant varieties of Segre-Veronese varieties with three or more factors. Let $\mathbf{n}=(n_0, \dots, n_{k-1})$ be a $k$-tuple of positive integers and let $X_{\mathbf{n},\mathbf{d}}$ be the Segre-Veronese variety $\prod_{i=0}^{k-1} \P^{n_i}$ embedded in multi-degree $\mathbf{d} = (d_0, \dots, d_{k-1})$. 
In~\cite{AB}, we generalized the so-called m\'ethode d'Horace diff\'erentielle, which was first introduced in \cite{AH0} by Alexander and Hirschowitz for studying secant varieties of Veronese varieties,  to Segre-Veronese varieties. This differential Horace method allows one to check whether $\sigma_s(X_{\mathbf{n},\mathbf{d}})$ has the expected dimension by induction on $\mathbf{n}$ and $\mathbf{d}$. In order to apply the Horace method, one needs initial cases regarding either dimensions or degrees. It is therefore vital to classify defective secant varieties of Segre-Veronese varieties embedded in lower multi-degree including $(1,\dots, 1,d)$.  A significant step toward the completion of such a classification problem is to detect previously unknown defective cases. 
The main purpose of this paper is to show that certain three-factor and four-factor Segre-Veronese varieties $X_{\mathbf{n},\mathbf{d}}$ with $\mathbf{d}=(1, \dots, 1,d)$, $d \geq 2$,  have defective secant varieties, most of which were not known before.  It is worth noting that some of these defective cases have already been discovered by Catalisano, Geramita and Gimigliano in~\cite{CGG} and also by Carlini and Catalisano~\cite{CaCa}. The examples discussed in this paper can thus be naturally viewed as an extension of their results.

The idea of the proof is to find a non-singular rational subvariety  $C_{s-1}$ of $X_{\mathbf{n},\mathbf{d}}$ passing through $s$ generic points of $X_{\mathbf{n},\mathbf{d}}$. The existence of such a subvariety provides an upper bound of $\dim \sigma_s(X_{\mathbf{n},\mathbf{d}})$; namely,  
\[
 \dim \sigma_s(X_{\mathbf{n},\mathbf{d}}) \leq s\left(\dim X_{\mathbf{n},\mathbf{d}}  -\dim C_{s-1} \right) + \dim \langle C_{s-1} \rangle, 
\]
where $\langle C_{s-1} \rangle$ is the linear span of $C_{s-1}$. 
We will then show that if $k \in \{3,4\}$, there are infinitely many $(\mathbf{n}, s)$ such that the right hand side of the above inequality is strictly smaller than
\[
\min \left\{s\left(\sum_{i=0}^{k-1} n_i+1\right)-1, \ {n_{k-1}+d \choose d}\prod_{i=0}^{k-2}(n_i+1)
-1\right\}.
\]
This establishes the defectivity of surprisingly many $\sigma_s(X_{\mathbf{n},\mathbf{d}})$.

\section{Preliminaries}
\label{sec:preliminaries}
For a given non-negative integer $k$, 
let $\mathbf{x}$ be an ordered list of $k$ integers. We write $\mathbf{x}[i]$ for the $i^\mathrm{th}$ element of $\mathbf{x}$ for every $i \in \{0, \dots, k-1\}$. We sometimes write $\mathbf{x}[-i]$ instead of $\mathbf{x}[k-i]$. This is useful if the number of elements in $\mathbf{x}$ is not specified. For a non-empty subset  $\{i_0, \dots, i_{a-1}\}$ of $\{0, \dots, k-1\}$, 
\[
\mathbf{x}\left[\{i_0, \dots, i_{a-1}\}\right] = \{\mathbf{x}[j] \ | \ j \in  \{i_0, \dots, i_{a-1}\}\}.
\]  
We also write $\mathbf{x}\left[i_0, \dots, i_{a-1}\right]$ for $\mathbf{x}\left[\{i_0, \dots, i_{a-1}\}\right]$. 
Thus $\mathbf{x}[i]$ also means the one-element list consisting  of the $i^{\mathrm{th}}$ element of $\mathbf{x}$. 
Let $i$ be a non-negative integer and let $a$ be a positive integer with $i < k-a$. 
Then  
\[
\mathbf{x}[i,\dots, -a] = \{ \mathbf{x}[j] \ | \  i \leq j \leq k-a\}.
\] 
For two ordered lists  of  $k$ integers  $\mathbf{x}$ and $\mathbf{y}$,  we say that $\mathbf{x} \geq \mathbf{y}$ if $\mathbf{x}[i] \geq \mathbf{y}[i]$ for every $i \in \{0, \dots, k-1\}$ and we denote 
by $\mathbf{x}  - \mathbf{y}$ the difference of two vectors. 

Define a function $N : (\Z_{\geq 0})^k \times (\Z_{\geq 0})^k \rightarrow \N$ by
\[
 N(\mathbf{n}, \mathbf{d}) = \prod_{i=0}^{k-1} {\mathbf{n}[i]+ \mathbf{d}[i] \choose \mathbf{d}[i]} 
\]
and 
a function $D: (\Z_{\geq 0})^k  \rightarrow \N$ by 
\[
 D(\mathbf{n} ) = \sum_{i=0}^{k-1} \mathbf{n}[i]. 
\]

Let $K$ denote an algebraically closed field of characteristic $0$ and  let $V$ be an $(n+1)$-dimensional vector space over $K$ with basis $\{e_0, \dots, e_n\}$. Then we denote by 
\begin{itemize}
\item[-] $V^* = \mathrm{Hom}_K(V,K)$ the dual of $V$, 
\item[-] $\{e_0^*, \dots, e_n^*\}$ a dual basis for $V^*$ with relation $e_i^*(e_j) = \delta_{ij}$,  
\item[-] $\Sym_dV$ the $d^{\mathrm{th}}$ symmetric power of $V$, 
\item[-] $\P V$ the projective space of $V$, 
\item[-] $[v] \in \P V$ the equivalence class containing $v \in V \setminus \{0\}$,  
\item[-] $\langle S \rangle$ the linear span of $S$ if $S$ is a subset of  $\P V$,  
\item[-] $\nu_d$ the $d^{\mathrm{th}}$ Veronese map from $\P V $ to $\P \Sym_dV$, i.e., the map given by sending $[v]$ to $[v^d]$. 
\item[-] $\widehat X$ the affine cone over a variety $X \subset \P V$ in $V$. 
\item[-] $T_p(X)$ the projective tangent space to a variety $X \subset \P V$ at $p \in X$. 
\end{itemize}

Let $k$ be a positive integer, let $\mathbf{n}$ and $\mathbf{d}$ be two $k$-tuples of positive integers.  For each $i \in \{0, \dots, k-1\}$, let $V_i$ be an $(\mathbf{n}[i]+1)$-dimensional vector space over $K$ with basis $S_i = \left\{e_{i,0}, \dots, e_{i,\mathbf{n}[i]}\right\}$. Then we call an element of the $N(\mathbf{n},\mathbf{d})$-dimensional vector space $\bigotimes_{i=0}^{k-1} \Sym_{\mathbf{d}[i]} V_i$ a {\it partially symmetric tensor}, because it is invariant under the natural action of the symmetric group on $S_i$ for each $i \in \{0, \dots, k-1\}$. 

A partially symmetric tensor $t  \in \bigotimes_{i=0}^{k-1} \Sym_{\mathbf{d}[i]} V_i$ is said to have {\it rank one} if $t$ is of the form $u_0^{\mathbf{d}[0]} \otimes \cdots \otimes u_{k-1}^{\mathbf{d}[k-1]}$ for some $u_i \in V_i$. We say that $t$ has {\it rank $s$} if $t$ can be written as a linear combination of $s$ rank-one partially symmetric tensors, but not fewer. In other words, $t$ has rank $s$ if it lies in the linear span of $s$ linearly independent rank one partially symmetric tensors in $\bigotimes_{i=0}^{k-1} \Sym_{\mathbf{d}[i]} V_i$. 

Now consider the map from $\prod_{i=0}^{k-1} \Sym_{\mathbf{d}[i]} V_i$ to $\bigotimes_{i=0}^{k-1}\Sym_{\mathbf{d}[i]} V_i$ given by sending $\left(u_0^{\mathbf{d}[0]}, \cdots, u_{k-1}^{\mathbf{d}[k-1]}\right)$ to $u_0^{\mathbf{d}[0]} \otimes \cdots \otimes u_{k-1}^{\mathbf{d}[k-1]}$. This  induces the embedding of $\prod_{i=0}^{k-1} \nu_{\mathbf{d}[i]}(\P V_i)$ to  $\P \left( \bigotimes_{i=0}^{k-1}\Sym_{\mathbf{d}[i]} V_i\right)$. The image of $\prod_{i=0}^{k-1} \nu_{\mathbf{d}[i]}(\P V_i)$ under this  embedding is called a {\it Segre-Veronese variety} and denoted by $\mathrm{Seg}\left(\prod_{i=0}^{k-1} \nu_{\mathbf{d}[i]}(\P V_i)\right)$ or $X_{\mathbf{n}, \mathbf{d}}$.  Note that the affine cone over the $D(\mathbf{n})$-dimensional variety  $X_{\mathbf{n}, \mathbf{d}}$ in $\bigotimes_{i=0}^{k-1}\Sym_{\mathbf{d}[i]} V_i$ is the set of non-zero rank-one partially symmetric tensors. 

Recall that a secant $(s-1)$-plane to $X_{\mathbf{n}, \mathbf{d}}$ is the linear span of $s$ linearly independent points in $X_{\mathbf{n}, \mathbf{d}}$. By definition, it is clear that points which lie in a secant $(s-1)$-plane to $X_{\mathbf{n}, \mathbf{d}}$ correspond to rank-$s$ partially symmetric tensors in $\bigotimes_{i=0}^{k-1}\Sym_{\mathbf{d}[i]} V_i$. Thus $\sigma_s(X_{\mathbf{n}, \mathbf{d}})$ can be viewed as the projectivization of the closure of the set of partially symmetric tensors of rank at most $s$. 

In order to check the defectivity of $\sigma_s(X_{\mathbf{n}, \mathbf{d}})$, we need to determine whether the equality  
\[
\dim \sigma_s(X_{\mathbf{n}, \mathbf{d}}) =  \min \left\{ s \left[D(\mathbf{n})+1\right], \ N(\mathbf{n}, \mathbf{d}) \right\}-1
\]
holds.  To find the dimension of $\sigma_s(X_{\mathbf{n}, \mathbf{d}})$, it is sufficient to  compute the dimension of the tangent space $T_q [\sigma_s(X_{\mathbf{n}, \mathbf{d}})]$ to $\sigma_s(X_{\mathbf{n}, \mathbf{d}})$ at a generic point $q$.
To do so,  it is useful to use the following classical theorem called Terracini's lemma: 
\begin{theorem}[Terracini's Lemma] 
\label{th:terracini}
Let $p_1,\dots, p_s$ be generic points of $X_{\mathbf{n}, \mathbf{d}}$ and let $q$ be a generic point $q$ of $\langle p_1, \dots, p_s\rangle$.  Then 
\[
 T_q [\sigma_s(X_{\mathbf{n}, \mathbf{d}})] = \left\langle T_{p_1} (X_{\mathbf{n}, \mathbf{d}}), \dots, T_{p_s} (X_{\mathbf{n}, \mathbf{d}})\right\rangle. 
 \]
\end{theorem}
\begin{remark}
\label{rem:tangent}
Let $u_i, v_i \in V_i \setminus \{0\}$ for each $i \in \{0, \dots, k-1\}$. Consider the parametric curve given by $(u_0+\varepsilon v_0)^{\mathbf{d}[0]} \otimes \cdots \otimes (u_{k-1} + \varepsilon v_{k-1})^{\mathbf{d}[k-1]}$, for $\varepsilon \in K$. Taking the derivative of this parametric curve at $\varepsilon=0$ proves that the affine cone over the tangent space to $X_{\mathbf{n},\mathbf{d}}$ at $p = \left[u_0^{\mathbf{d}[0]} \otimes \cdots \otimes u_{k-1}^{\mathbf{d}[k-1]}\right]$ in $\bigotimes_{i=0}^{k-1} \Sym_{\mathbf{d}[i]} V_i$ is given by 
\[
 \widehat{T}_p (X_{\mathbf{n},\mathbf{d}}) = \sum_{i=0}^{k-1} u_0^{\mathbf{d}[0]} \otimes \cdots \otimes u_i^{\mathbf{d}[i]-1} V_i \otimes  \cdots \otimes u_{k-1}^{\mathbf{d}[k-1]}. 
\] 
Thus  $\widehat{T}_p (X_{\mathbf{n},\mathbf{d}})$ is represented by a $\left[\sum_{i=0}^{k-1} (n_i+1)\right] \times N(\mathbf{n}, \mathbf{d})$ matrix. In order to compute the dimension of $T_q [\sigma_s(X_{\mathbf{n}, \mathbf{d}})]$, we select $s$ points $p_0, \dots, p_{s-1} \in X_{\mathbf{n}, \mathbf{d}}$, construct the matrix representing $\widehat{T}_{p_i} (X_{\mathbf{n},\mathbf{d}})$, stack these matrices and then compute the rank of the resulting matrix. If the rank of this matrix is equal to $\min \{ s(D(\mathbf{n})+1), \ N(\mathbf{n}, \mathbf{d}) \}$, then we can conclude that $\sigma_s(X_{\mathbf{n}, \mathbf{d}})$ has the expected dimension by semi-continuity. 
\end{remark}
\begin{remark}
\label{rem:fat}
For simplicity, we write $\P^{\mathbf{n}}$ for the multi-projective space $\prod_{i=0}^{k-1} \P^{\mathbf{n}[i]}$. Note that $H^0\mathcal{O}_{\P^{\mathbf{n}}}(\mathbf{d})$ can be identified with the set of hyperplanes in $\P^{N(\mathbf{n},\mathbf{d})-1}$. Since the condition that a hyperplane $H\subset \P^{N(\mathbf{n},\mathbf{d})-1}$ contains $T_p(X_{\mathbf{n}, \mathbf{d}})$ is equivalent to the condition that $H$ intersects $X_{\mathbf{n}, \mathbf{d}}$ in the first infinitesimal neighborhood of $p$, the elements of $H^0 \mathcal{I}_{p^2}(\mathbf{d})$ can be viewed as hyperplanes containing 
$T_p(X_{\mathbf{n}, \mathbf{d}})$, where $p^2$ denotes the double point supported at $p$. Let $Z$ be a collection of $s$ double points on 
$\P^{\mathbf{n}}$ and let $\sI_Z$ be its ideal sheaf. Terracini's lemma implies that $\dim \sigma_s(X_{\mathbf{n},  \mathbf{d}})$ equals to the value of the Hilbert function $h_{\P^{\mathbf{n}}}(Z, \cdot)$ of $Z$ at $\mathbf{d}$, i.e.,
\[
h_{\P^{\mathbf{n}}}(Z,\mathbf{d})=\dim H^0\ko_{\P^{\mathbf{n}}}(\mathbf{d})-
\dim H^0\sI_Z(\mathbf{d}).
\]
To prove that  $\sigma_s(X_{\mathbf{n}, \mathbf{d}})$ has the expected
dimension is equivalent to prove that
\[
h_{\P^{\mathbf{n}}}(Z,\mathbf{d})= \min \left\{
 s\left[D(\mathbf{n}) +1\right], \ N(\mathbf{n},\mathbf{d}) \right\}.
\]
\end{remark}
\section{Numeric conditions for $\sigma_s(X_{\mathbf{n}, \mathbf{d}})$ to be defective}
\label{sec:numeric}
Let $k$ be a non-negative integer, let $\Lambda = \{0, \cdots, k-1\}$,  let $\mathbf{n} \in \N^k$ and let $\mathbf{d} \in \N^k$ such that   $\mathbf{d}[j] = 1$ for every $j \in \{0, \dots, k-2\}$ and $\mathbf{d}[k-1] \geq 2$. 
Consider a $k$-tuple $\mathbf{e}_0$ of non-negative integer such that 
$\mathbf{e}_0 \leq \mathbf{d}$ and $1 \leq \mathbf{e}_0[k-1] \leq \mathbf{d}[k-1]-1$.  
We set $\mathbf{e}_1=\mathbf{d}-\mathbf{e}_0$. 
For each $i \in \Z_2$, let $\Lambda_i = \left\{j \in \Lambda \  | \  \mathbf{e}_i[j] \not=0 \right\}$. 
Let $\varphi \in \bigotimes_{j \in \Lambda}  \Sym_{\mathbf{d}[j]} V_j$. 
For each $i \in \Z_2$, let $\Gamma_i(\varphi)$ be the corresponding contraction by $\varphi$: 
\[
\Gamma_i(\varphi) : \left(\bigotimes_{j \in \Lambda_{i-1}} \Sym_{\mathbf{e}_{i-1}[j]} V_j\right)^*
\rightarrow \bigotimes_{j \in \Lambda_i} \Sym_{\mathbf{e}_i[j]} V_j
\]
Note  that if $\rank \varphi=1$,  then $\rank \Gamma_i(\varphi) =1$. Thus if $\varphi$ has rank $s$, then the rank of $\Gamma_i(\varphi)$ is less than or equal to $s$.  
By definition, $\Gamma_0(\varphi)=\Gamma_1(\varphi)^T$. 

For each $i \in \Z_2$, let $A_i$ denote $\im \Gamma_i(\varphi)$ and let $B_i$ denote $\ker \Gamma_{i-1}(\varphi)$. Suppose that $\varphi$ is a  rank-$s$ partially symmetric tensor 
in $\bigotimes_{j \in \Lambda} \Sym_{\mathbf{d}[j]} V_j$ and can be written as a linear combination of $s$ rank-one partially symmetric tensors $\varphi_a$,  
$0 \leq a \leq s-1$,  in $\bigotimes_{j \in \Lambda} \Sym_{\mathbf{d}[j]} V_j$.  Then $A_i$ is the subspace of $\bigotimes_{j \in \Lambda_i} \Sym_{\mathbf{e}_i[j]} V_j$ spanned by the $\im \Gamma_i(\varphi_a)$'s.  Note that $\dim \im \Gamma_i(\varphi_a)=1$ for each $0\leq a\leq k-1$.
Hence if $\varphi$ is sufficiently general and if
\begin{equation}
\label{new}
s\leq \dim\left(\bigotimes_{j \in \Lambda_i} \Sym_{\mathbf{e}_i[j]} V_j\right)=N(\mathbf{n},\mathbf{e}_i).
\end{equation}  
then  $A_i$ has dimension $s$. 

For each $i \in \Z_2$, we define $C_{i,s-1}$ as  
\[
C_{i,s-1}=\P A_i \cap \Seg \left(\prod_{j \in \Lambda_i}  \nu_{\mathbf{e_i}[j]} (\P V_j)\right).
\]
If $\varphi$ is sufficiently general and  if
\begin{eqnarray}
\label{eq:ineq0}
\max_{i\in \Z_2} \left\{N(\mathbf{n}, \mathbf{e}_i) - \sum_{j \in \Lambda_i} \mathbf{n}[j]+1\right\} \leq 
s,
\end{eqnarray}
then $C_{i,s-1}$ is a non-degenerate, non-singular variety of dimension 
\begin{eqnarray*}
 && \dim \P A_i +\dim \Seg \left(\prod_{j \in \Lambda_i}  \nu_{\mathbf{e_i}[j]} (\P V_j)\right) -\dim 
\left\langle\P A_i \cup \Seg \left(\prod_{j \in \Lambda_i}  \nu_{\mathbf{e_i}[j]} (\P V_j)\right)\right\rangle 
\\
 &= & s-1+\dim \mathrm{Seg}\left(\prod_{j \in \Lambda_i} \nu_{\mathbf{e_i}[j]} (\P V_j)\right) 
 - \dim \P\left(\bigotimes_{j \in \Lambda_i} \Sym_{\mathbf{e}_i[j]} V_j\right)\\
& = & 
s+\sum_{j \in \Lambda_i}\mathbf{n}[j] - N(\mathbf{n}, \mathbf{e}_i) \\ 
& \geq & 1. 
\end{eqnarray*}
Let us define
\begin{eqnarray*}
E(\mathbf{n}) & = &  \sum_{i \in \Z_2}\left\{\sum_{j \in \Lambda_i}\mathbf{n}[j] -N(\mathbf{n},\mathbf{e}_i)+s \right\}-\mathbf{n}[k-1] \\
& = & D(\mathbf{n}) - \sum_{i \in \Z_2} (N(\mathbf{n},\mathbf{e}_i)-s). 
\end{eqnarray*} 
Suppose that 
\begin{eqnarray}
&&\mathbf{n}[\Lambda_i[0]] \geq  N(\mathbf{n},\mathbf{e}_i)-s>0  \label{eq:ineq1}
\end{eqnarray}
for each $i\in\Z_2$. 
Then we obtain  the inequalities \eqref{new} and 
\begin{eqnarray}
\label{eq:ineq2}
D(\mathbf{n}) > E(\mathbf{n}). 
\end{eqnarray}
\begin{proposition}
\label{prop:rnc} 
Suppose that Inequalities~(\ref{eq:ineq0}) and (\ref{eq:ineq1})   
are satisfied. For each $a \in \{0, \dots, s-1\}$, let $[\varphi_a]$ be a generic  point of $X_{\mathbf{n},\mathbf{d}}$. Then there is a proper non-singular subvariety of $X_{\mathbf{n}, \mathbf{d}}$ of dimension $E(\mathbf{n})$ passing through $[\varphi_0], \dots, [\varphi_{s-1}]$. 
\end{proposition}
\begin{proof}
Let $[\varphi] = \left[\sum_{a=0}^{s-1} \varphi_a\right]$.  Recall that, for each $i \in \Z_2$, $B_i$ denotes $\ker \Gamma_{i-1}(\varphi)$. Because of the choice of $\varphi$, the rank of $\Gamma_i(\varphi)$ is $s$, and hence $B_i$ has dimension $N(\mathbf{n},\mathbf{e}_{i})-s$ for each $i \in \Z_2$.  
Recall also that $1 \leq \mathbf{e}_0[k-1] \leq \mathbf{d}[k-1]-1$ and $\mathbf{e}_1= \mathbf{d}-\mathbf{e}_0$. Thus $k-1 \in \Lambda_i$ for each $i \in \Z_2$. 
For a given rank-one partially symmetric tensor in $\bigotimes_{j \in \Lambda_i\setminus \Lambda_i[0]} \Sym_{\mathbf{e}_i[j]} V_j$, consider the associated contraction  
\[
\left(\bigotimes_{j \in \Lambda_i} \Sym_{\mathbf{e}_i[j]} V_j\right)^* \rightarrow \left(\Sym_{\mathbf{e}_i[\Lambda_i[0]]} V_{\Lambda_i[0]} \right)^*= V_{\Lambda_i[0]}^* . 
\]
Note that the last equality holds, because we are assuming  $\mathbf{e}_i[\Lambda_i[0]]=1$. The image of $B_i$ under this contraction is a subspace $U_i^*$ of $V_{\Lambda_i[0]}^*$ with dimension at most $N(\mathbf{n},\mathbf{e}_i)-s$. So Condition~(\ref{eq:ineq1}) implies that,  for every $\left[\ell^{\mathbf{e}_i[k-1]}\right] \in \nu_{\mathbf{e}_i[k-1]} (\P V_{k-1})$, there exists a rank-one  tensor $t$ in $\bigotimes_{j \in \Lambda_i\setminus \{k-1\}} \Sym_{\mathbf{e}_i[j]} V_j = \bigotimes_{j \in \Lambda_i\setminus \{k-1\}}  V_j$ such that $[t \otimes \ell^{\mathbf{e}_i[k-1]}] \in C_{i,s-1}$. 

For each $i \in \Z_2$, we define a map $\pi_i: C_{i,s-1} \rightarrow \P V_{k-1}$ by sending $\left[t \otimes \ell^{\mathbf{e}_i[k-1]}\right] $ to $[\ell]$. Then we have proved that $\pi_i$ is onto. The fiber of this map is 
\[
\Seg\left(\P (U_i^*)^\perp  \times \prod_{j \in \Lambda_i[1, \dots, -2]} \nu_{\mathbf{e}_i[j]} (\P V_j)\right) = \Seg\left(\P (U_i^*)^\perp  \times \prod_{j \in \Lambda_i[1, \dots, -2]} \P V_j \right) ,
\]
and so it has dimension
\begin{eqnarray*}
&& \dim V_{\Lambda_i[0]}^*- \dim  U_i^*-1+\sum_{j \in \Lambda_i[1, \dots, -2]}\mathbf{n}[j]  \\ 
 & = &    \mathbf{n}[\Lambda_i[0]]+1-N(\mathbf{n},\mathbf{e}_i)+s-1+\sum_{j \in \Lambda_i[1, \dots, -2]}\mathbf{n}[j]  \\
 & = &\sum_{j \in \Lambda_i[0, \dots, -2]}\mathbf{n}[j]-N(\mathbf{n},\mathbf{e}_i)+s . 
\end{eqnarray*}

Consider now the fiber product of $C_{0,s-1}$ and $C_{1,s-1}$ over $\P V_{k-1}$: 
\[
 C_{0,s-1} \times_{\P V_{k-1}} C_{1,s-1} = 
 \left\{ ([x],[y])  \in C_{0,s-1} \times  C_{1,s-1}\  |\  \pi_0([x]) = \pi_1([y]) 
 \right\}. 
\]
For simplicity, we write $C_{s-1}$ for $C_{0,s-1} \times_{\P V_{k-1}} C_{1,s-1}$. 
If $([x],[y]) \in C_{s-1}$, then, since $\pi_0([x]) = \pi_1([y])$, the tensor product of $x$ and $y$, which lies in $\bigotimes_{i \in \Z_2} \bigotimes_{j \in \Lambda_i }\Sym_{\mathbf{e}_i[j]} V_j$,  must lie in $\bigotimes_{j \in \Lambda} \Sym_{\mathbf{d}[j]} V_j$. Thus the Segre map 
\[
 \prod_{i \in \Z_2} \prod_{j \in \Lambda_i} \nu_{\mathbf{e}_i[j]}\left(\P V_j \right)  
 \rightarrow \P\left(\bigotimes_{i \in \Z_2} \bigotimes_{j \in \Lambda_i }\Sym_{\mathbf{e}_i[j]} V_j\right) 
\]
restricted to $C_{s-1}$ embeds  $C_{s-1}$ in $X_{\mathbf{n},\mathbf{d}}$. Since the fiber of the projection from $C_{s-1}$ to $\P V_{k-1}$ over each $[\ell] \in \P V_{k-1}$ has dimension $\sum_{j \in \Lambda_i[0, \dots, -2]}\mathbf{n}[j]-N(\mathbf{n},\mathbf{e}_i)+s$, we obtain 
\begin{eqnarray*}
 \dim C_{s-1} & = &
  \sum_{i \in \Z_2} \left\{ \sum_{j \in \Lambda_i[0, \dots, -2]}\mathbf{n}[j]-N(\mathbf{n},\mathbf{e}_i)+s\right\} +  \mathbf{n}[k-1] \\
 & = &   \sum_{i \in \Z_2} \left\{ \sum_{j \in \Lambda_i }\mathbf{n}[j]-N(\mathbf{n},\mathbf{e}_i)+s\right\} -  \mathbf{n}[k-1]  \\
 & = & E(\mathbf{n}). 
\end{eqnarray*}
The smoothness of $C_{s-1}$ follows immediately from the construction and the smoothness of the $C_{i,s-1}$'s. We have therefore completed the proof. 
\end{proof}
By the definitions of the $\mathbf{e}_i$'s, for any $j \in \Lambda \setminus \{k-1\}$, there is a 
unique pair of index $(i,a)$ 
such that $j = \Lambda_i[a]$. 
Let $\Phi$ be the map 
\[
\Phi: \prod_{i \in \Z_2} \prod_{j \in \Lambda_i} \Sym_{\mathbf{e}_i[j]} V_j^* \rightarrow 
\bigotimes_{j \in \Lambda} \Sym_{\mathbf{d}[j]} V_j^*
\]
such that the $j^{\mathrm{th}}$ component $\Phi(w)[j]$ is given by
\[
\Phi(w)[j] = \left\{
\begin{array}{ll}
w[\Lambda_i[a]] & \quad \mbox{if $j = \Lambda_i[a]$;}\\
w[\Lambda_0[-1]] w[\Lambda_1[-1] & \quad \mbox{if $j=k-1$.} 
\end{array}
\right.
\]
We write $F(\mathbf{n}, \mathbf{e}_0, \mathbf{e}_1)$ for 
\[
 s [ D(\mathbf{n})-E(\mathbf{n})] + N(\mathbf{n},\mathbf{d})-1-\sum_{i \in \Z_2} 
 N(\mathbf{n},\mathbf{e}_{i-1})[N(\mathbf{n},\mathbf{e}_i)-s]+\prod_{i \in \Z_2} [N(\mathbf{n},\mathbf{e}_i)-s]. 
\]
\begin{theorem}
\label{th:dim}
Suppose that Inequalities~(\ref{eq:ineq0}) and  (\ref{eq:ineq1}) are satisfied.  
Then 
\[
 \dim \sigma_s(X_{\mathbf{n},\mathbf{d}}) \leq 
 F(\mathbf{n}, \mathbf{e}_0, \mathbf{e}_1). 
\]
\end{theorem}
\begin{proof}
We keep the same notation as in the proof of Proposition~\ref{prop:rnc}.  
By construction, $C_{s-1}$ is contained in the linear subspace $H\subseteq \P \left( \bigotimes_{j \in \Lambda} \Sym_{\mathbf{d}[j]} V_j \right)$ defined by the ideal generated by 
$\sum_{i \in \Z_2} \Phi\left[B_i \times \prod_{j \in  \Lambda_{i-1}} \Sym_{\mathbf{e}_{i-1}[j]} V_j^*\right]$. Since  this is the quotient vector space of $\bigoplus_{i \in \Z_2} \Phi\left[B_i \times \prod_{j \in  \Lambda_{i-1}} \Sym_{\mathbf{e}_{i-1}[j]} V_j^*\right]$ modulo $\Phi(B_0 \times B_1)$ and $\dim \Phi(B_0 \times B_1) \leq \dim B_0 \otimes B_1 = \prod_{i \in \Z_2} [N(\mathbf{n},\mathbf{e}_i)-s]$,  it follows from the assumption that 
\begin{eqnarray*}
\dim H  & \leq  & N(\mathbf{n},\mathbf{d})-1-\left\{\dim \bigoplus_{i \in \Z_2} \Phi\left[B_i \times \prod_{j \in  \Lambda_{i-1}} \Sym_{\mathbf{e}_{i-1}[j]} V_j^*\right] - \dim \Phi(B_0 \times B_1) \right\}\\ 
 & \leq & N(\mathbf{n},\mathbf{d})-1-\sum_{i \in \Z_2} 
 N(\mathbf{n},\mathbf{e}_{i-1})[N(\mathbf{n},\mathbf{e}_i)-s]+\prod_{i \in \Z_2} [N(\mathbf{n},\mathbf{e}_i)-s]. 
\end{eqnarray*}
Thus it follows from Terracini's lemma 
and Proposition~\ref{prop:rnc} that the dimension of  $\sigma_s(X_{\mathbf{n},\mathbf{d}})$ is bounded from above by 
\begin{eqnarray*}
 \dim  \sigma_s(X_{\mathbf{n},\mathbf{d}}) & \leq & 
 s[\dim X_{\mathbf{n}, \mathbf{d}} - \dim C_{s-1}] + \dim \langle C_{s-1} \rangle \\
 & \leq & s[\dim X_{\mathbf{n}, \mathbf{d}} - E(\mathbf{n})]+ \dim H \\
 & \leq & F(\mathbf{n}, \mathbf{e}_0, \mathbf{e}_1), 
\end{eqnarray*}
which completes the proof. 
\end{proof}
As an immediate consequence, we obtain the following corollary: 
\begin{corollary}
\label{cor:defect}
Suppose that Conditions~(\ref{eq:ineq0}) and (\ref{eq:ineq1})
are satisfied and that 
\begin{eqnarray}
\label{eq:ineq3}
F(\mathbf{n},\mathbf{e}_0, \mathbf{e}_1) < \min \{s[D(\mathbf{n})+1], N(\mathbf{n},\mathbf{d}) \}-1. 
\end{eqnarray}
Then $\sigma_s(X_{\mathbf{n}, \mathbf{d}})$ is defective.  
\end{corollary}
\begin{remark}
\label{rem:superabundant}
It is straightforward to verify that 
the following equality holds for any $\mathbf{n}$, $\mathbf{d}$, $\mathbf{e}_0$, $\mathbf{e}_1$ and $s$:
\begin{eqnarray}
\label{eq:superabundant}
N(\mathbf{n}, \mathbf{d}) -1 - F(\mathbf{n}, \mathbf{e}_0, \mathbf{e}_1)  = \left[N(\mathbf{n},\mathbf{e}_0)-s\right]\left[N(\mathbf{n},\mathbf{e}_1)-s\right].
\end{eqnarray}
Hence if (\ref{eq:ineq1}) holds, then  $N(\mathbf{n}, \mathbf{d}) -1 - F(\mathbf{n}, \mathbf{e}_0, \mathbf{e}_1) >0$ by (\ref{eq:superabundant}).  
From Corollary~\ref{cor:defect} it follows therefore that  if  $\mathbf{n}$, $\mathbf{d}$, $\mathbf{e}_0$, $\mathbf{e}_1$ and $s$ satisfy (\ref{eq:ineq0}) and (\ref{eq:ineq1}) and if 
\begin{equation}\label{ineq-new}
s \geq \lceil N(\mathbf{n},\mathbf{d})/(D(\mathbf{n})+1) \rceil, 
\end{equation}
then $\sigma_s(X_{\mathbf{n}, \mathbf{d}})$ is defective.   
\end{remark}
\begin{example}
For an arbitrary positive integer $a$ with $a \geq 2$, consider the following two cases:
\begin{itemize}
\item $\mathbf{n}=(1,1,a)$, $\mathbf{d}=(1,1,2)$, $s=2a+1$;
\item $\mathbf{n}=(a,a,2)$, $\mathbf{d}=(1,1,2)$, $s=3a+2$.
\end{itemize}
In each of these cases, let $\mathbf{e}_0=(1,0,1)$ and $\mathbf{e}_1=(0,1,1)$.  
It is immediate to check that \eqref{eq:ineq0}, \eqref{eq:ineq1} and \eqref{ineq-new} hold. 
Hence, by Remark \ref{rem:superabundant}, we can conclude that 
$\sigma_s(X_{\mathbf{n}, \mathbf{d}})$ is defective.
Note that the defective cases listed above were first discovered by Catalisano, Geramita and Gimigliano (see \cite{CGG} for more details).
\end{example}
\begin{theorem}
\label{th:three-factor} 
Let $\mathbf{n} = (n,n+a,1)$, let $\mathbf{d} = (1,1,d)$ with $d \geq 2$, $\mathbf{e}_0 = (1,0,d-e)$ with $1 \leq e \leq d-1$ and let $\mathbf{e}_1 = (0,1,e)$. If $\mathbf{n}$, $\mathbf{d}$, $\mathbf{e}_0$, $\mathbf{e}_1$ and $s$ satisfy (\ref{eq:ineq0}) and (\ref{eq:ineq1}), then $\sigma_s(X_{\mathbf{n}, \mathbf{d}})$ is defective. 
\end{theorem}
\begin{proof}
By Corollary~\ref{cor:defect} and Remark~\ref{rem:superabundant}, it is sufficient to show that 
\[
s[D(\mathbf{n}) +1]-1-F(\mathbf{n}, \mathbf{e}_0, \mathbf{e}_1)  \ge 0.
\]
To do so, we will prove the following equality:
\[
s[D(\mathbf{n}) +1]-1-F(\mathbf{n}, \mathbf{e}_0, \mathbf{e}_1)  =  \left[s- N(\mathbf{n},\mathbf{e}_0)+\sum_{j \in \Lambda_0} \mathbf{n}[j]\right]\left[s- N(\mathbf{n},\mathbf{e}_1)+\sum_{j \in \Lambda_1} \mathbf{n}[j]\right]. 
\]
Then the conclusion follows immediately from \eqref{eq:ineq0}.

By (\ref{eq:superabundant}) one can easily obtain
\[
s[D(\mathbf{n}) +1]-1-F(\mathbf{n}, \mathbf{e}_0, \mathbf{e}_1)  = 
[s-N(\mathbf{n},\mathbf{e}_0)][s-N(\mathbf{n},\mathbf{e}_1)]-N(\mathbf{n}, \mathbf{d}) +s[D(\mathbf{n})+1].  
\]
We thus need to show 
\begin{eqnarray*}
&& [s-N(\mathbf{n},\mathbf{e}_0)][s-N(\mathbf{n},\mathbf{e}_1)]-N(\mathbf{n}, \mathbf{d}) +s[D(\mathbf{n})+1] \\
& = & 
\left[s- N(\mathbf{n},\mathbf{e}_0)+\sum_{j \in \Lambda_0} \mathbf{n}[j]\right]\left[s- N(\mathbf{n},\mathbf{e}_1)+\sum_{j \in \Lambda_1} \mathbf{n}[j]\right]
\end{eqnarray*}
or equivalently 
\begin{eqnarray*}
& & s\left(\sum_{j \in \Lambda_0} n[j] + \sum_{j \in \Lambda_1} n[j] -D(\mathbf{n})-1\right) \\
& = & N(\mathbf{n}, \mathbf{e}_0) \sum_{j \in \Lambda_1} n[j] +N(\mathbf{n}, \mathbf{e}_1) \sum_{j \in \Lambda_0} n[j]  -  \sum_{j \in \Lambda_0} n[j]\sum_{j \in \Lambda_1} n[j] - N(\mathbf{n}, \mathbf{d}). 
\end{eqnarray*}
If $\mathbf{n} = (n,n+a,1)$, $\mathbf{d} = (1,1,d)$, $\mathbf{e}_0 = (1,0,d-e)$ and $\mathbf{e}_1 = (0,1,e)$, then 
\[
\sum_{j \in \Lambda_0} n[j] + \sum_{j \in \Lambda_1} n[j] -D(\mathbf{n})-1 = (n+1)+(n+a+1)-(2n+a+1)-1 = 0. 
\]
and 
\begin{eqnarray*}
& & N(\mathbf{n}, \mathbf{e}_0) \sum_{j \in \Lambda_1} n[j] +N(\mathbf{n}, \mathbf{e}_1) \sum_{j \in \Lambda_0} n[j]  -  \sum_{j \in \Lambda_0} n[j]\sum_{j \in \Lambda_1} n[j] - N(\mathbf{n}, \mathbf{d}) \\
& = & (n+1)(d-e+1)(n+a+1)+(n+a+1)(e+1)(n+1) \\ 
&& -(n+1)(n+a+1)-(n+1)(n+a+1)(d+1) \\
& = & 0.
\end{eqnarray*}
So the desired equality holds, and hence we completed the proof.  
\end{proof}
In order to find explicit examples of defective secant varieties by applying Theorem \ref{th:three-factor},  one has to solve, for given $d \geq 2$ and $1 \leq e \leq d-1$,  the system of inequalities  (\ref{eq:ineq0}) and (\ref{eq:ineq1}) for $a$ and $s$. 
Below we will discuss  two simpler cases. 
\begin{example}
\label{ex:even}
For given $n, d \in \N$,  
let  $a \in \{0, \dots, \lceil n/d \rceil-1\}$, let $\mathbf{n} = (n,n+a,1)$ and let $\mathbf{d}=(1,1,2d)$. 
Consider $\mathbf{e}_0=(1,0,d)$ and $\mathbf{e}_1= (0,1,d)$. 
Then $\Lambda_0 = \{0,2\}$ and $\Lambda_1 = \{1,2\}$, and we obtain 
\[
 N(\mathbf{n},\mathbf{e}_i) - \sum_{j \in \Lambda_i} \mathbf{n}[j] +1 = 
 \left\{ 
 \begin{array}{ll}
 (n+1)d+1 & \mbox{if $i=0$;} \\
 (n+a+1)d+1 & \mbox{if $i=1$.} 
 \end{array}
 \right. 
\]
Thus 
\[
\max_{i \in \Z_2} \left\{N(\mathbf{n},\mathbf{e}_i) - \sum_{j \in \Lambda_i} \mathbf{n}[j] +1\right\}
=  (n+a+1)d+1. 
\]
For a given $k \in \{1, \dots, n-ad\}$, let $s = (n+a+1)d+k$. Then $s$ clearly satisfies Condition~(\ref{eq:ineq0}). 
Moreover since 
\[
\begin{array}{ccccccccc}
\mathbf{n}[\Lambda_0[0]]  & = &  n & \geq & N(\mathbf{n},\mathbf{e}_0)-s & = &  n-ad-k+1 & > &  0 \\ 
\mathbf{n}[\Lambda_1[0]]  & = &  n+a & \geq & N(\mathbf{n},\mathbf{e}_1)-s & = &  n+a-k+1 & > &  0,
\end{array}
\]
(\ref{eq:ineq1}) is also satisfied.  
Therefore, it follows from Theorem~\ref{th:three-factor} that $\sigma_s(X_{\mathbf{n}, \mathbf{d}})$ is defective.  
\end{example}
%
\begin{example}
Let  $a$ be a non-negative integer. 
For given $n, d \in \N$, let $\mathbf{n} = (n,n+a,1)$ and let $\mathbf{d}=(1,1,2d+1)$. 
Consider $\mathbf{e}_0=(1,0,d+1)$ and $\mathbf{e}_1= (0,1,d)$. 
Then $\Lambda_0 = \{0,2\}$ and $\Lambda_1 = \{1,2\}$, and we obtain 
\[
 N(\mathbf{n},\mathbf{e}_i) - \sum_{j \in \Lambda_i} \mathbf{n}[j] +1 = 
 \left\{ 
 \begin{array}{ll}
 (n+1)(d+1)+1 & \mbox{if $i=0$;} \\
 (n+a+1)d+1 & \mbox{if $i=1$.} 
 \end{array}
 \right. 
\]
Now consider the following two cases: (i) $a \in \{0, \dots, \lfloor (n+1)/d \rfloor\}$ and (ii) $a \in \left\{\lfloor (n+1)/d \rfloor+1, \dots, \lfloor 2n/d \rfloor \right\}$.  

\vspace{2mm}
\noindent 
(i) Let $a \in \{0, \dots, \lfloor (n+1)/d \rfloor\}$. Then  
\[
\max_{i \in \Z_2} \left\{N(\mathbf{n},\mathbf{e}_i) - \sum_{j \in \Lambda_i} \mathbf{n}[j] +1\right\}
=  (n+1)(d+1)+1. 
\]
For a given $k \in \{1, \dots, \min\{n+1,a(d+1)\}-1\}$, let $s = (n+1)(d+1)+k$. Then $s$ clearly satisfies Condition~(\ref{eq:ineq0}). Moreover since 
\[
\begin{array}{ccccccccc}
\mathbf{n}[\Lambda_0[0]]  & = &  n & \geq & N(\mathbf{n},\mathbf{e}_0)-s & = &  n+1-k & > &  0 \\ 
\mathbf{n}[\Lambda_1[0]]  & = &  n+a & \geq & N(\mathbf{n},\mathbf{e}_1)-s & = &  a(d+1)-k & > &  0,
\end{array}
\]
then Inequalities (\ref{eq:ineq1}) are also satisfied. Thus the defectivity of $\sigma_s(X_{\mathbf{n}, \mathbf{d}})$ follows from Theorem~\ref{th:three-factor}.

\vspace{2mm}
\noindent 
(ii) Let $a \in \left\{\lfloor (n+1)/d \rfloor+1, \dots, \lfloor 2n/d \rfloor \right\}$.  Then  
\[
\max_{i \in \Z_2} \left\{N(\mathbf{n},\mathbf{e}_i) - \sum_{j \in \Lambda_i} \mathbf{n}[j] +1\right\}
=  (n+a+1)d+1. 
\]
Let $k \in \{1, \dots, \min \{n+a, 2n-ad+1\}\}$ and let $s = (n+a+1)d+k$. Then one can show that Conditions (\ref{eq:ineq0}) and (\ref{eq:ineq1}) are satisfied in the same way as in (i). 
The secant variety $\sigma_s(X_{\mathbf{n}, \mathbf{d}})$ is therefore defective by Theorem~\ref{th:three-factor}. 
\end{example}
There are infinitely many defective secant varieties of Segre-Veronese varieties with four factors, 
which we describe below. 
\begin{theorem}
For given positive integers $n$ and $d$ with $d \geq 2$, let 
$-1 \leq k < (nd-3n+d-2)/(2n+1)$,   
let $\mathbf{n} = (1,n, dn+d+k,1)$, let $\mathbf{d} = (1,1,1,d)$ and let $s = 2d(n+1)-1$.  
Then $\sigma_s(X_{\mathbf{n}, \mathbf{d}})$ is defective. 
\end{theorem}
\begin{proof}
Let $\mathbf{e}_0 = (1,1,0,d-1)$ and let $\mathbf{e}_1 = (0,0,1,1)$. Then $\Lambda_0 = \{0,1,3\}$ and $\Lambda_1 = \{2,3\}$. 
Since $N(\mathbf{n},\mathbf{e}_0) = 2d(n+1)$, $N(\mathbf{n},\mathbf{e}_1) = 2[d(n+1)+k+1]$ 
and $k\leq d(n+1)-3$ for every $k \in\{-1, (nd-3n+d-2)/(2n+1)-1\}$, we obtain
\[
\max_i \left\{N(\mathbf{n},\mathbf{e}_i)- \sum_{j \in \Lambda_i}\mathbf{n}[j]+1\right\} \leq s. 
\]
Thus $s$ satisfies (\ref{eq:ineq0}). 
We also have 
\[
\mathbf{n}[\Lambda_i[0]] - [N(\mathbf{n},\mathbf{e}_i) - s] =  
\left\{
\begin{array}{ll} 
0 & \mbox{if $i = 0$,} \\
d(n+1)-k-3 & \mbox{if $i = 1$,}
\end{array}
\right.
\]
and $N(\mathbf{n},\mathbf{e}_i) > s$ for each $i\in\Z_2$, and hence (\ref{eq:ineq1}) is satisfied. 
Note that 
$N(\mathbf{n},\mathbf{d}) \geq s[D(\mathbf{n})+1]$ if and only if $(nd-3n+d-5)/(2n+3) \leq k$. 
Now it is easy to check the following inequalities hold:
\[
N(\mathbf{n}, \mathbf{d}) -1 - F(\mathbf{n}, \mathbf{e}_0, \mathbf{e}_1)  =  2k+3 > 0 
\]
\[
 s[D(\mathbf{n}) +1]-1-F(\mathbf{n}, \mathbf{e}_0, \mathbf{e}_1)  =  - (2n+1)k + nd - 3n + d - 2 >0 
\]
for every  $k \in \{-1, (nd-3n+d-2)/(2n+1)-1\}$. 
Hence Condition \eqref{eq:ineq3} is also satisfied,  and thus 
it follows from Corollary~\ref{cor:defect} that $\sigma_s(X_{\mathbf{n}, \mathbf{d}})$ does not have the expected dimension. 
\end{proof}

\noindent 
{\it Acknowledgements}. 
We thank Maria Virginia Catalisano for sharing with us the list of  known defective secant varieties of Segre-Veronese varieties, which was the starting point of this research.  


\end{document}